\documentclass{amsart}
\usepackage{graphicx, amssymb, eucal, hyperref, color, amsmath, amsthm}
\newtheorem{iThm}{Theorem}

\newtheorem{thm}{Theorem}[section]
\newtheorem{cor}[thm]{Corollary}
\newtheorem{lem}[thm]{Lemma}
\newtheorem{fact}[thm]{Fact}

\theoremstyle{definition}
\newtheorem{defn}[thm]{Definition}

\newtheorem*{convention}{Convention}
\theoremstyle{remark}
\newtheorem{rem}[thm]{Remark}

\newcommand{\norm}[1]{\left\Vert#1\right\Vert}

\newcommand{\op}[1]{\operatorname{#1}}

\renewcommand{\epsilon}{\varepsilon}
\renewcommand{\phi}{\varphi}
\begin{document}
\title{Model theoretic properties of dynamics on the Cantor set}%
\author[C. J. Eagle]{Christopher J. Eagle} 
\address[C. J. Eagle]{University of Victoria, Department of Mathematics and Statistics, PO BOX 1700 STN CSC, Victoria, British Columbia, Canada, V8W 2Y2}%
\email{eaglec@uvic.ca}
%\urladdr{http://www.math.uvic.ca/~eaglec}

\author[A. Getz]{Alan Getz}
\address[A. Getz]{University of Victoria, Department of Mathematics and Statistics, PO BOX 1700 STN CSC, Victoria, British Columbia, Canada, V8W 2Y2}%
\email{agetz@uvic.ca}
\date{\today}%
% ----------------------------------------------------------------
\begin{abstract}
We examine topological dynamical systems on the Cantor set from the point of view of the continuous model theory of commutative C*-algebras.  After some general remarks we focus our attention on the generic homeomorphism of the Cantor set, as constructed by Akin, Glasner, and Weiss.  We show that this homeomorphism is the prime model of its theory.  We also show that the notion of ``generic" used by Akin, Glasner, and Weiss is distinct from the notion of ``generic" encountered in Fra\"iss\'e theory.
\end{abstract}
\maketitle
% ----------------------------------------------------------------
\section*{Introduction}\label{sec:Introduction}
Studying well-understood structures equipped with a new symbol for an automorphism is a common theme in contemporary model theory.  Especially in the context of fields, this study has led to striking applications of model-theoretic methods to other parts of mathematics, such as Hrushovski's celebrated proof of the Manin-Mumford conjecture in number theory \cite{Hrushovski}.

In this paper we take the theme of studying the model theory of an automorphism on a structure and apply it to the setting of topological spaces.  That is, given a topological space $X$ and a homeomorphism $S : X \to X$, we wish to study the pair $(X, S)$ from a model-theoretic point of view.  Applying model theory to topological spaces is not entirely straightforward, because topological spaces are not easily made into the kind of structures studied in model theory.  One approach, developed by Bankston, is to ``dualize" model-theoretic notions into topological notions, working without any syntax (see \cite{Bankston1987} and its successors).  Another method is to find a model-theoretic structure that can act as a replacement for the space $X$.  For example, if $X$ is compact and $0$-dimensional then, by Stone duality, the Boolean algebra of clopen subsets of $X$ can replace $X$ for most purposes.  However, Banaschewski \cite{Banaschewski1984} has shown that there is no class of classical model-theoretic structures, in any finite language, which is dual to the class of all compact spaces in the way that Boolean algebras are dual to compact $0$-dimensional spaces.  There has been some success using lattice bases as stand-ins for compact spaces (see, for instance, \cite{Dow2001} and \cite{Bartosova2011}), but the connection is imperfect, as a single space may have several different lattice bases that are not even elementarily equivalent.

The introduction of \emph{continuous first-order logic} by Ben Yaacov and Usvyatsov in \cite{BenYaacov2010} provides another method for approaching topological spaces.  By Gel'fand duality the class of compact Hausdorff spaces is dual to the class of commutative unital C*-algebras, and this latter class can be treated using the tools of continuous logic.  We assume that the reader is familiar with continuous logic for metric structures, as described in \cite{BenYaacov2008a}.  The structures we consider will all be expansions of C*-algebras.  Although C*-algebras do not directly fit into the framework of \cite{BenYaacov2008a} (because they are not bounded), a suitable way of using continuous logic for the model-theoretic study of C*-algebras was introduced by Farah, Hart, and Sherman in \cite{Farah2014a}.  The model theory of C*-algebras has since been extensively developed; though we will not make direct use of most of these recent results, we nevertheless recommend that the reader consult the survey \cite{FarahEtAl} for more information about the model theory of C*-algebras.  The use of C*-algebras to study the model theory of topological spaces has been previously explored in \cite{EagleVignati}, \cite{EagleThesis}, \cite{Eagle2016b}, and \cite{Goldbring2019}.

Broadly, we are interested in studying structures of the form $(C(X), \sigma)$, where $X$ is a compact Hausdorff space and $\sigma$ is the automorphism of $C(X)$ induced by a homeomorphism $S : X \to X$.  Following the viewpoint that the structure $\mathcal{M}$ should be well-understood before attempting to study it with an automorphism, we will primarily restrict ourselves to the case where $X$ is the Cantor set.  The model theory of $C(\op{Cantor~set})$ is quite well understood.  The theory of this structure can be explicitly axiomatized as the theory of commutative unital real rank $0$ C*-algebras without non-trivial minimal projections.  This theory is separably categorical, has quantifier elimination, and is the model completion of the theory of commutative unital C*-algebras (see \cite{EagleVignati} or \cite{EagleThesis}).  On the other hand, the study of Cantor set dynamical systems is a rich area, and such systems exhibit an interesting diversity of properties (see \cite{CMS} for a survey).

Section 1 contains preliminaries and fixes notation.  Starting with Section 2 we focus our attention on the case of $C(\op{Cantor~set})$.  In contrast to fields, where adding operators of various kinds often produces theories with a model companion (see \cite{Moosa}), we prove:

\begin{iThm}[Theorem \ref{thm:NoCompanion}]
The theory of $C(\op{Cantor~set})$ with an automorphism does not have a model companion.
\end{iThm}

For the specific case of odometers on the Cantor set we prove that the theory is an extremely strong invariant:

\begin{iThm}[Theorem \ref{thm:Odometers}]
If two odometers on the Cantor set are elementarily equivalent then they are topologically conjugate.
\end{iThm}

However, this does not give rise to a separably categorical theory of Cantor dynamical systems, because we show that being an odometer is not an axiomatizable property.

Kechris and Rosendal \cite{KechrisRosendal} showed that there is a unique dense $G_\delta$ conjugacy class in the group of homeomorphisms of the Cantor set, and a concrete description of the homeomorphisms in this class was given by Akin, Glasner, and Weiss \cite{AGW}. In Section 3 we study the model theoretic properties of these \emph{generic homeomorphisms}.  The main results are:

\begin{iThm}[Theorems \ref{thm:NonaxiomatizableLifting}, \ref{thm:OmittingTypes}, \ref{thm:Prime}]
Being a generic homeomorphism is not axiomatizable, but is expressible as an omitting types property.  If $S$ is the generic homeomorphism of the Cantor set, then $(C(\op{Cantor~set}), S)$ is the prime model of its theory.
\end{iThm}

It is interesting that Kechris and Rosendal used Fra\"iss\'e limits (of discrete structures) in their proof of the existence of the generic homeomorphism.  To conclude the paper we give a Fra\"iss\'e construction of another homeomorphism of the Cantor set.  This homeomorphism is also ``generic" in a model-theoretic sense, but we show that it is \emph{not} the generic one constructed by Kechris and Rosendal, thus also showing that the notion of ``generic" arising from this Fra\"iss\'e limit is not the usual notion of genericity in the group of homeomorphisms.

\subsection*{Acknowledgements}
We are deeply grateful to Ian Putnam for several very helpful discussions about Cantor set dynamical systems.  The second author was financially supported by both an NSERC USRA and a JCURA award at the University of Victoria.

\section{Preliminaries}
Throughout this paper our context will be a compact metrizable space $X$ together with a homeomorphism $S : X \to X$.  Via Gel'fand duality the pair $(X, S)$ corresponds to the pair $(C(X), \sigma)$, where $C(X)$ is the commutative unital C*-algebra of continuous complex-valued functions on $X$, and $\sigma : C(X) \to C(X)$ is the automorphism of $C(X)$ defined by $\sigma(f) = f \circ S$.  We view $(C(X), \sigma)$ as a structure in the language of C*-algebras together with a new unary function symbol for representing $\sigma$.  Formally, in the setting of \cite{FarahEtAl}, we actually need several function symbols to represent $\sigma$ on the various sorts making up the structure $C(X)$.  Since $\sigma$ is an isometry there is no difficulty in determining either the appropriate target sorts or the moduli of uniform continuity for these new function symbols, so we will avoid dwelling on this issue here, and instead treat $\sigma$ as if it were represented by a single function symbol.

For the remainder of the paper, let $\mathcal{L}_{C^*}$ denote the language of unital C*-algebras, and let $\mathcal{L}$ be $\mathcal{L}_{C^*}$ together with a new unary function symbol $\sigma$, as described above.  There is an $\mathcal{L}_{C^*}$-theory $T_{C^*}$ such that models of $T_{C^*}$ are exactly the commutative unital C*-algebras, and by adding the axioms expressing that $\sigma$ is an automorphism we obtain the theory $T_{C^*-A}$ whose models are structures of the form $(M, \sigma)$ where $M$ is a commutative unital C*-algebra and $\sigma$ is an automorphism of $M$.

The following fact was observed in \cite{EagleVignati}:
\begin{fact}
There is an $\mathcal{L}_{C^*}$-theory $T_{Cantor}$ whose models are precisely the algebras $C(X)$ where $X$ is a compact $0$-dimensional Hausdorff space without isolated points.
\end{fact}

Let $T_{Cantor-A} = T_{Cantor} \cup T_{C^*-A}$, so the models of $T_{Cantor-A}$ are of the form $(C(X), \sigma)$ where $\sigma$ is an automorphism and $X$ is a compact $0$-dimensional space without isolated points.  The Cantor set $2^{\mathbb{N}}$ is, up to homeomorphism, the unique compact metrizable $0$-dimensional space with no isolated points.  Thus, up to isomorphism, the only separable model of $T_{Cantor}$ is $C(2^{\mathbb{N}})$, and consequently every separable model of $T_{Cantor-A}$ is isomorphic to one of the form $(C(2^{\mathbb{N}}), \sigma)$ for some automorphism $\sigma$.

\begin{convention}
For the remainder of the paper, we will use the same symbol (typically $\sigma$) to denote both an autohomeomorphism of a space $X$ and the induced automorphism of the C*-algebra $C(X)$.
\end{convention}

\begin{defn}
Throughout this paper, if $\mathcal{C}$ is a class of compact metrizable spaces and $P$ is a property of topological dynamical systems, we say that $P$ is \emph{separably axiomatizable for $\mathcal{C}$} if there is a theory $T$ of $\mathcal{L}_{C*-A}$ such that separable models of $T$ are of the form $(C(X), \sigma)$, where $X \in \mathcal{C}$ and $\sigma$ satisfies $P$.
\end{defn}

\section{The theory of $C(2^\mathbb{N})$ with an automorphism}
We now restrict our attention to the theory $T_{Cantor-A}$; that is, we consider the case of dynamical systems on compact $0$-dimensional spaces without isolated points.  The theory $T_{Cantor}$ has quantifier elimination \cite{EagleVignati}.  After adding the symbol for the automorphism, things become substantially more complicated.

\begin{thm}\label{thm:NoCompanion}
$T_{Cantor-A}$ has no model companion.  In particular, $T_{Cantor-A}$ is not model complete and does not have quantifier elimination.
\end{thm}

\begin{proof}
The result follows from the fact that the theory of Boolean algebras with an automorphism does not have a model companion \cite{kikyomodelcompanions}.  By using the Keisler-Shelah theorem (both the version for continuous logic and discrete logic) all of the properties required to express being a model companion can be reformulated in terms of ultraproducts.  An ultraproduct can be described as a colimit of products (see \cite{EagleThesis} for details).  The functor from commutative unital real rank 0 C*-algebras without minimal projections to atomless Boolean algebras obtained by compositing Gel'fand duality with Stone duality preserves products and colimits, so preserves ultraproducts.  Combined, these facts imply that if $T^*$ was a model companion of $T_{Cantor-A}$ then it would correspond to a model companion of the theory of atomless Boolean algebras with an automorphism, which does not exist.
\end{proof}

\subsection{Odometers}
We recall the important class of \emph{odometers} on the Cantor set.  If $G$ is a countable residually finite group, and $(G_n)_{n=1}^{\infty}$ are a decreasing sequence of normal subgroups with $\bigcap_{n=1}^{\infty}G_n = \{1_G\}$, then $G$ acts naturally on the inverse limit of the system $G/G_1 \leftarrow G/G_2 \leftarrow \cdots$.  Actions of this kind are variously called either $G$-odometers or profinite completions of $G$, and they arise naturally in both topological dynamics and ergodic theory (see \cite{Putnam}, \cite{Grigorchuk}, \cite{Cortez}).  As we are interested in the dynamics of a single homeomorphism we will restrict ourselves to the case $G=\mathbb{Z}$, in which case we have a more concrete description of odometers (see \cite{CMS} for more details):
\begin{defn}
Let $(a_n)_{n \geq 1}$ be a sequence of integers, each greater than or equal to $2$, and for each $n$ let $A_n = \lbrace 0,...,a_n -1 \rbrace$.  Let $X = \Pi A_n$ endowed with the product topology.  Define $\sigma : X \rightarrow X$ as follows.  Let $(b_n)_{n \geq 1} \in \Pi A_n$.  If $b_1 < a_1 - 1$, let $\sigma(b_n)_{n \geq 1} = (b_1 + 1, b_2, b_3, ...)$.  If $b_1 = a_1 - 1$, we let $\sigma(b_n)_{n \geq 1} = (0,...,0, b_n + 1, b_{n+1}, b_{n+2},...)$, where $n$ is the smallest number such that $b_n < a_n - 1$, or $\sigma(b_n)_{n \geq 1} = (0,0,0,...)$ if no such number $n$ exists.  The map $\sigma$ may also be characterized as addition by $1$ on the left with carry over (resembling something like an odometer in one's car, except that addition is on the left rather than the right).  An odometer is any Cantor system isomorphic to $(\Pi A_n, \sigma)$ for some sequence $(a_n)_{n \geq 1}$.
\end{defn}

\begin{thm}
There are non-elementarily equivalent odometers.  In particular, the theory $T_{Cantor-A}$ is not complete.
\end{thm}
\begin{proof}
Let $X=\{0, 1\}^\mathbb{N}$, and let $\sigma$ be the normal odometer map on $X$; that is, the map corresponding to the constant sequence $(2,2,2,...)$.  Let $Y=\{0, 1, 2\}^\mathbb{N}$, and let $\tau$ be the normal odometer on $Y$.  Then $X$ and $Y$ are Cantor sets, and both $(X, \sigma)$ and $(Y, \sigma)$ are odometers, so it suffices to show that $(C(X), \sigma) \not\equiv (C(Y), \tau)$.  In $(C(X), \sigma)$ there is a projection $x$ such that $x \neq 0, x\neq 1$, and $\sigma(x) = 1-x$.  The existence of such an $x$ can be expressed as a sentence on $L_{C^*-A}$.  We claim that $(C(Y),\tau)$ does not satisfy this sentence.

If $(C(Y), \tau)$ satisfied the sentence above, then we could find a non-trivial projection $z \in \operatorname{proj}(Y)$ whose support is a cylinder set that is a subset of the support of $x$, since cylinder sets form a basis for the topology on $Y$.  Then $z$ returns to itself after $3^n$ iterations of $\tau$, for some $n$.  But this is not compatible with the above condition, since if $\tau(x) = 1-x$ then $\tau^n(z) = z$ only when $n$ is even.
\end{proof}

We next show that elementarily equivalent odometers are isomorphic.  However, this does not give an example of a separably categorical theory of Cantor set dynamical systems, because we will also show that being an odometer is not an axiomatizable property.

\begin{lem}[\cite{CMS}]
The odometers corresponding to the sequences $(a_n)_{n \geq 1}$ and $(b_n)_{n \geq 1}$ are conjugate if and only if for every $N$ there exists an $M$ such that $a_1 \cdot ... \cdot a_N$ divides $b_1 \cdot ... \cdot b_M$, and for every $K$ there exists $L$ such that $b_1 \cdot ... \cdot b_K$ divides $a_1 \cdot ... \cdot a_L$.
\end{lem}

\begin{thm}\label{thm:Odometers}
Suppose that $(X, \sigma)$ and $(Y, \tau)$ are odometers on Cantor sets $X$ and $Y$.  Then $(C(X), \sigma) \equiv (C(Y), \tau)$ if and only if $(C(X), \sigma) \cong (C(Y), \tau)$.
\end{thm}
\begin{proof}
The `if' direction is clear.  To prove the `only if' direction, let $(X,\sigma)$ be the odometer for the sequence of integers $(a_n)_{n \geq 1}$ and $(Y,\tau)$ the odometer for $(b_n)_{n \geq 1}$ and suppose $(C(X),\sigma) \equiv (C(Y),\tau)$.  We will show that $(X, \sigma)$ is topologically conjugate to $(Y, \tau)$, from which the conclusion follows.

For each $n \geq 1$ let $m_n = a_1 \cdot \ldots \cdot a_n$.  For any $k$ let $\phi_k$ be the sentence that asserts there are projections $x_1,\ldots,x_k$, all of disjoint supports which sum to $1$, and the automorphism sends $x_1 \mapsto x_2, \ldots ,x_{k-1} \mapsto x_k, x_k \mapsto x_1$.  Then $\forall n \geq 1$, $(C(X),\sigma) \models \phi_{m_n}$, so $(C(Y), \tau) \models \phi_{m_n}$.  Then we can find $m_n$ projections in $(C(Y),\tau)$ that satisfy $\phi_{m_n}$.  The supports of these projections, $U_1, \ldots ,U_{m_n}$, are clopen in $Y$ and covered by a collection of cylinder sets.  Each $U_i$ is compact, so there are disjoint cylinder sets $C^i_1,\ldots,C^i_{k_i}$ such that for $1 \leq i \leq m_n$, $U_i = \bigcup_{1 \leq j \leq k_i} C^i_j$.  We can also assume each of the $C$'s are all cylinder sets of finite sequences of the same length $l$, since we can break up any cylinder set into a disjoint union of cylinder sets corresponding to longer sequences.  Since $\tau(U_1) = U_2,\ldots,\tau(U_{m_n}) = U_1$, and $\tau$ sends each cylinder set to another set, we have that $k_1 = \ldots = k_{m_n}$.  Therefore $m_n$ divides the product $b_1 \cdot \ldots \cdot b_l$.  A similar argument shows that for any $N$ there is an $M$ such that $b_1 \cdot \ldots \cdot b_N$ divides $a_1 \cdot \ldots \cdot a_M$.  Thus $(X,\sigma)$ and $(Y,\tau)$ are topologically conjugate, so $(C(X),\sigma) \cong (C(Y), \tau)$.  
\end{proof}

\begin{thm}
The property of being an odometer on the Cantor set is not separably axiomatizable.
\end{thm}
\begin{proof}
Suppose there is a theory $T$ whose separable models are exactly the odometers.  For each prime number $p$ let $(X_p,\sigma_p)$ be the odometer for the constant sequence $(p,p,p,\ldots)$.  Then for every prime $p$, $(C(X_p),\sigma_p) \models T$.  Let $\mathcal{U}$ be a nonprincipal ultrafilter on the set of prime numbers, and let $(A,\sigma)$ be the ultraproduct $(A, \sigma) = \prod_{\mathcal{U}}(C(X_p), \sigma_p)$.  Then $(A,\sigma) \models T$ by \L o\'s theorem.

Let us say that a dynamical system $(Y,\tau)$ has the property $(*)$ if there is a refining sequence of partitions of $Y$, each partition consisting of clopen sets, such that $\tau$ permutes the sets in these partitions in the same way as an odometer would with its cylinder sets.  That is, there is a sequence $(a_n)_{n \geq 1}$ of integers such that each $a_n \geq 2$, and a sequence $(P_n)_{n \geq 1}$ such that $P_n$ is a set of $a_1 \cdot \ldots \cdot a_n$ disjoint clopen sets in $Y$ refining $P_{n-1}$ such that for every $A \in P_{n-1}$ there are $a_n$ elements $B_1,\ldots,B_{a_n}$ of $P_n$ such that $A = B_1 \cup \ldots \cup B_{a_n}$ and $\tau(B_1) = B_2,\ldots,\tau(B_{a_n}) = B_1$.  Notice that this property is satisfied by every odometer.

We claim that $(A,\sigma)$ does not satisfy the property $(*)$.  Suppose it did.  Then we could find $(a_n)_{n \geq 1}$ and $(P_n)_{n \geq 1}$ as in the above paragraph.  Then we have a set of projections $P_1 = \{ x_1,\ldots,x_{a_1} \}$ in $A$, where $\sigma(x_1) = x_2, \ldots, \sigma(x_{a_1 -1}) = x_{a_1}, \sigma(x_{a_1}) = x_1$.  Since being a projection is a definable property, there is therefore some sequence $(y_2,y_3,y_5,y_7,\ldots)$ in the equivalence class of $x_1$, where each $y_p \in X_p$, such that $\mathcal{U}$-almost every $y_p$ is a projection such that $\sigma_p^{a_1}(y_p) = y_p$.  That is, 
\[P = \{p : p \text{ is prime and } y_p \text{ is a projection and } \sigma_p^{a_1}(y_p) = y_p\} \in \mathcal{U}.\]
Since for each $p \in P$, $\sigma_p^{a_1} (y_p) = y_p$, we must have that $p|a_1$.  But $P$ is infinite, so this is impossible.  Thus, $(A,\sigma)$ does not satisfy $(*)$.  By L\"owenheim-Skolem, find $(A_0,\sigma_0) \preceq (A,\sigma)$ which is separable.  Then $(A_0,\sigma_0)$ cannot satisfy $(*)$ either, since if it did then the sequence of projections making $(*)$ true for $(A_0,\sigma_0)$ would also make $(*)$ true for $(A,\sigma)$.  Hence, $(A_0,\sigma_0)$ cannot be an odometer.  On the other hand, $(A_0, \sigma_0) \models T$ because $(A, \sigma) \models T$, giving a contradiction.
\end{proof}

\section{The generic homeomorphism of $2^{\mathbb{N}}$}
We now turn our attention to a very specific kind of homeomorphism of the Cantor set.  Let $\mathcal{H}(2^\mathbb{N})$ be the group of homeomorphisms of the Cantor set, and equip $\mathcal{H}(2^\mathbb{N})$ with the topology of pointwise convergence.  Kechris and Rosendal \cite{KechrisRosendal} showed that $\mathcal{H}(2^\mathbb{N})$ has a dense $G_\delta$ conjugacy class.  Akin, Glasner and Weiss \cite{AGW} subsequently found an explicit description of this conjugacy class by constructing a homeomorphism $\rho : 2^\mathbb{N} \to 2^\mathbb{N}$ whose conjugacy class in $\mathcal{H}(2^\mathbb{N})$ is dense $G_\delta$.  Moreover, they isolated a key property, the \emph{lifting property}, and showed that the homeomorphisms satisfying the lifting property are exactly those that are topologically conjugate to the homeomorphism $\rho$ that they constructed.  For our purposes we do not need the details of the construction of $\rho$, as we will work entirely with the lifting property.

\subsection{Defining the lifting property}
Let $l$, $r$, and $m$ be symbols, and for $n \geq 1$ let $L_n = \{l\} \times \{0,\ldots, n! - 1\}$, $\widetilde{R}_n = \{r\} \times \{0,\ldots, n! - 1\}$, and $M_n=\{m\}\times\{-n+1,\ldots,n-1\}$.  Let $S_n = L_n \cup M_n \cup \widetilde{R}_n$.  We put a $\ \widetilde{} \ $ over the $R$ to distinguish it from the relation $R_n$ on $S_n$, which we now define.  

For $0 \leq k < n! - 1$ we let $(l,k)R_n(l,k+1)$, $(r,k)R_n(r,k+1)$, $(l,n!-1)R_n(l,0)$, and $(r,n!-1)R_n(r,0)$.  We also require that for $-n+1 \leq k < n - 1$, $(m,k)R_n(m,k+1)$, and $(l,0)R_n(m,-n + 1)$ and $(m,n - 1)R_n(r,0)$.  A pair $(S_n,R_n)$ is called an \textit{n-spiral}.  The picture one should have in mind is two copies of $n!$ points related to one another cyclically, connected by $2n-1$ points related linearly.  The reader is encouraged to draw a picture of the case where $n = 2$ and $3$ to gain a full understanding of the structure of these sets.

We now define three functions from $(S_{n+1},R_{n+1})$ to $(S_{n},R_{n})$.  For any symbol $a \in \{ l, m, r\}$, define $\xi_L$ and $\xi_G$ by 
\[\xi_L(a,k) = (l,k (\textrm{mod} \ n!)) \qquad \text{ and } \qquad \xi_G(a,k) = (r,k (\textrm{mod} \ n!)).\]
 The last map, $\xi_M$, is defined as follows:
 
 \[\xi_M(a, k) = \begin{cases} (a, k (\textrm{mod} \ n!)) &\text { if $a \in \{l, r\}$} \\ (m, k) &\text{ if $a=m$ and $-n+1\leq k \leq n-1$} \\ (l, 0) & \text { if $a=m$ and $k= -n$} \\ (r, 0) & \text{ if $a=m$ and $k=n$}\end{cases}\]
 
We think of the map $\xi_M$ as projecting the points in $(S_{n+1},R_{n+1})$ `straight down' onto $(S_{n},R_{n})$.  It can be checked that if two points are related by $R_{n+1}$, then their images under each of these functions are related under $R_n$.

Let $W_n$ be the union of the set of $6^n$ $n$-spirals indexed by the set of words of length $n$ built from the alphabet $\{ \xi_{L_1}, \xi_{L_2}, \xi_{M_1}, \xi_{M_2}, \xi_{G_1}, \xi_{G_2} \}$.  We give $W_n$ the discrete topology.  We also let $R_n$ denote the union of the relations on each of the copies of $S_n$, so that $R_n$ denotes a relation on $W_n$.  

We define the map $\xi : W_{n+1} \rightarrow W_n$ as follows.  For a point $x \in W_n$, let $a$ be the last letter in the word indexing the $(n+1)$-spiral from which $x$ was taken.  If $a = \xi_{L_1}$ or $\xi_{L_2}$ then $\xi(x) = \xi_L(x)$, if $a = \xi_{M_1}$ or $\xi_{M_2}$ then $\xi(x) = \xi_M(x)$, and if $a = \xi_{G_1}$ or $\xi_{G_2}$ then $\xi(x) = \xi_G(x)$.  By composition of these maps we may define a canonical map from $W_m$ to $W_n$, for any $n > n$, which we also denote by $\xi$.

For a Cantor system $(X,\sigma)$, when we write $\phi : (X, \sigma) \rightarrow (W_n,R_n)$ we mean that $\phi$ is a continuous surjection from $X$ to $W_n$ and for any $x \in X$, $\phi(x)R_n\phi(\sigma(x))$.  Also, if $Y$ is a finite discrete space (not necessarily $W_n$ for some $n$) and $\phi : X \rightarrow Y$ is a continuous surjection, then we let $A_\phi:=\{\phi^{-1}\{ y \} \vert$ $ y \in Y \}$.  The \emph{mesh size} of $A_\phi$ is the maximum of all the diameters of the clopen sets in $A_\phi$.

We are finally ready to define the lifting property for Cantor dynamical systems, which characterizes the generic homeomorphism up to topological conjugacy.  

\begin{defn}
We say that a Cantor system $(X,\sigma)$ has the \emph{lifting property} if whenever $\epsilon > 0$ and $\phi : (X,\sigma) \rightarrow (W_n,R_n)$ is a continuous surjection for some $n$, there exists $k \geq 1$ and a continuous surjection $\psi : (X, \sigma) \rightarrow (W_{n+k},R_{n+k})$ such that $A_\psi$ has mesh size less than $\epsilon$ and $\phi = \xi \circ \psi$.
\end{defn}

The definition of the lifting property makes reference to a specific metric on $X$, and hence is not suitable for use in our setting.  We therefore reformulate it in purely topological terms, which we will then be able to use in the context of commutative C*-algebras.

\begin{lem}\label{lem:LiftingPropertyTopological}
Let $X$ be a Cantor set and $\sigma$ be a homeomorphism on $X$.  Then $(X,\sigma)$ has the Lifting Property if and only if $\forall n\in \mathbb{Z}^+ $, if $\phi:(X,\sigma)\rightarrow (W_n,R_n)$ is a continuous surjection, and $A$ is a refinement of $A_\phi$, then $\exists k \in \mathbb{Z}^+$ and $\exists \psi:(X,\sigma)\rightarrow(W_{n+k},R_{n+k})$ such that $A_\psi$ refines $A$ and $\phi = \xi \circ \psi$, where $\xi$ is the canonical map from $(W_{n+k},R_{n+k})$ to $(W_n,R_n)$.
\end{lem}

\begin{proof}
For the forward direction, suppose $(X,\sigma)$ has the lifting property, $\phi:(X,\sigma)\rightarrow(W_n,R_n)$ is a continuous surjection, and $A$ refines $A_\phi$. Let $\epsilon$ be less than the shortest distance between any two clopen sets in $A$ and also less than the mesh size of $A$.  Notice also that any partition of $X$ with mesh size less than or equal to $\epsilon$ cannot be equal to $A$, since $A$ has a larger mesh size, and moreover will refine $A$, since otherwise a clopen set from that partition would have elements from two separate sets from $A$, which would force the mesh size to be greater than $\epsilon$.  Since $(X,\sigma)$ has the lifting property, $\exists k$ and $\psi:(X,\sigma)\rightarrow(W_{n+k},R_{n+k})$ such that $\phi = \xi \circ \psi$ and $A_\psi$ has mesh less than $\epsilon$, thereby refining $A$.

The other direction follows immediately by choosing $A$ to have mesh size less than a given fixed $\epsilon$.
\end{proof}

\subsection{Non-axiomatizability of the lifting property}
\begin{lem}
Let $(X,\sigma)$ have the lifting property and suppose $\phi:(X,\sigma)\rightarrow(W_n,R_n)$ is a continuous surjection.  Then there is a sequence of partitions $(P_i)_{i\in\mathbb{N}}$, each containing the same number of clopen sets, such that $\forall i \in \mathbb{N}, P_i$ refines $A_\phi$, and there is no positive integer k such that for infinitely many $i \in \mathbb{N}$, $\exists\psi_i : (X,\sigma)\rightarrow (W_{n+k},R_{n+k})$ such that ${A_\psi}_i$ refines $P_i$ and $\phi = \xi \circ \psi_i$.
\end{lem}

\begin{proof}
Let $(X,\sigma)$ have the lifting property and let $\phi:(X,\sigma)\rightarrow(W_n,R_n)$ be a continuous surjection for some $n\in\mathbb{Z}^+$.  We construct a sequence of partitions of $X$, $(A_i)_{i\in\mathbb{N}}$, as follows.  Choose $\epsilon_0>0$, and let $A_0$ be the partition of $X$ obtained from finding a continuous surjection onto $W_{n+k_0}$, for some $k_0$, which composes with the canonical map $\xi : (W_{n+k_0}, R_{n+k_0}) \to (W_n, R_n)$ to give $\phi$, and where the mesh size of $A_0$ is less than $\epsilon_0$.  Let $\epsilon_1$ be half the mesh size of $A_0$, and repeat the process to find a refinement of $A_0$ with mesh size less than $\epsilon_1$, call it $A_1$, formed form a continuous surjection onto $W_{n +k_0+k_1}$, for some $k_1 \in \mathbb{Z}^+$, that collapses to the previous surjection onto $W_{n+k_0}$ when composed with $\xi$.  Continue in this way to obtain the sequence $(A_i)_{i \in \mathbb{N}}$.  Notice that as $i$ increases, the mesh size of $A_i$ tends to $0$.

We now define a sequence of clopen sets in $X$ as follows.  Choose a point, $a$, on the right side of one of the spirals in $W_n$.  Let $U_0$ be a clopen set in $A_0$ that is the pre-image of a point $a_0$ in the right-hand side of one of the finite spirals in $W_{n+k_0}$ that maps to $a$ under $\xi$.  Similarly, let $U_1$ be the pre-image of a point $a_1$ in the right hand side of a spiral in $W_{n + k_0 +k_1}$ that maps to $a_0$ under $\xi$.  Continue in this way to obtain a sequence $(U_i)_{i \in \mathbb{N}}$.  We observe a few properties of $(U_i)_{i \in \mathbb{N}}$.  First, as $i$ approaches infinity, the diameter of $U_i$ approaches zero.  Second, notice that $\forall i \in \mathbb{N}$, $\exists m \in \mathbb{Z}^+$ such that $\sigma^m (U_i) \cap U_i\neq\emptyset$.  Finally, since $\forall i \in \mathbb{N}$, $U_i\supseteq U_{i+1}$, by Cantor's Intersection Theorem, there is some element $b$ in $X$ such that $\bigcap_{i\in \mathbb{N}} U_i=\{ b \}$.  We are now ready to define our sequence of partitions $(P_i)_{i \in \mathbb{N}}$.

For each $i \in \mathbb{N}$, let $P_i = (A_\phi \backslash \{\phi ^{-1} \{ a \} \} ) \cup \{ U_i,\phi ^{-1} \{ a \} \backslash U_i \}$.  We claim that there is no $k \in \mathbb{Z} ^+$ such that for infinitely many $i \in \mathbb{N}$, $\exists \psi_i : (X, \sigma) \rightarrow (W_{n+k}, R_{n+k})$ such that $A_{\psi_i}$ refines $P_i$ and $\phi = \xi \circ \psi_i$.  To see why, let us suppose otherwise.  Then for some $k \in \mathbb{Z}^+$, for infinitely many $i \in \mathbb{N}$, $\exists \psi_i:(X,\sigma)\rightarrow(W_{n+k},R_{n+k})$ such that $\phi = \xi \circ \psi_i$ and $A_{\psi_i}$ refines $P_i$.  By passing to a relevant subsequence of $(P_i)_{i \in \mathbb{N}}$ we may assume without loss of generality that this statement holds for all $i$ in $\mathbb{N}$, not just some infinite subsequence.

For each $i \in \mathbb{N}$, let $V_i$ be such that $b \in V_i \in A_{\psi_i}$ (recall that $b$ satisfies $\bigcap_{i\in \mathbb{N}} U_i=\{ b \}$).  Notice that $\forall i \in \mathbb{N}$, $\exists j\geq i$ such that $U_j \subseteq V_i$.  However, we chose $U_j$ such that $\exists m \in \mathbb{Z}^+$ so that $\sigma^m (U_j) \cap U_j \neq \emptyset$.  This means that for each $i$, $V_i$ has a subset which returns to itself after a finite number of iterations of $\sigma$.  From this we may conclude that $\forall i \in \mathbb{N}$, $V_i$ is the pre-image of a point in either the left-hand or right-hand side of $W_{n+k}$.  It must be one of these cases for infinitely many $i$, so we define the subsequence $(V_{i_j})_{j \in \mathbb{N}}$  such that each $V_{i_j}$ is the pre-image under $\psi_{i_j}$ of a point on the same side of the spirals in $W_{n+k}$.  let us assume that each $V_{i_j}$ corresponds to some point on the right side of the spirals in $W_{n+k}$, but the rest of the proof can be adapted to the case where each $V_{i_j}$ corresponds to a point on the left side of $W_{n+k}$ by replacing each instance of $\sigma$ with $\sigma^{-1}$ from here on.

We have defined $(V_{i_j})_{j \in \mathbb{N}}$ such that $\bigcap_{j \in \mathbb{N}} V_{i_j} = \{ b \}$, but we have also defined each $V_{i_j}$ to be the pre-image under $\psi_{i_j}$ of a point on the right-hand side of one of the finite spirals in $W_{n+k}$, so $\forall j \in \mathbb{N}$, $\sigma^{(n+k)!} (V_{i_j}) \subseteq V_{i_j}$.  Notice also that $\forall j \in \mathbb{N}$, $\sigma^{(n+k)!} (b) \in \sigma^{(n+k)!} (V_{i_j})$, so $\sigma^{(n+k)!} (b) \in \bigcap_{j \in \mathbb{N}} \sigma^{(n+k)!} (V_{i_j})$, and thus $\{ \sigma^{(n+k)!} (b) \} \subseteq \bigcap_{j \in \mathbb{N}} \sigma^{(n+k)!} (V_{i_j}) \subseteq \bigcap_{j \in \mathbb{N}} V_{i_j} = \{ b \}$, so $b = \sigma^{(n+k)!} (b)$.  b is therefore a periodic point, which is not possible when $(X, \sigma)$ satisfies the lifting property.  Hence, we have arrived upon a contradiction.
\end{proof}

We are now prepared to show that the lifting property is not axiomatizable.

\begin{thm}\label{thm:NonaxiomatizableLifting}
If $(X, \sigma)$ is a Cantor system with the lifting property, then for any non-principal ultrafilter $\mathcal{U}$ on $\mathbb{N}$, the ultrapower $(C(X), \sigma)^{\mathcal{U}}$ corresponds to a dynamical system that does not have the lifting property.
\end{thm}
\begin{proof}
For any $n$, we may find a collection of projections which serve as indicator functions of the elements of $A_\phi$, where $\phi : (X, \sigma) \rightarrow (W_n,R_n)$ is a continuous surjection.  By the preceding lemma there is a sequence of further subdivisions of $(X, \sigma)$ for which there is no natural number $k$ such that, for more than finitely many of these subdivisions, we could lift to a further refinement of $(X, \sigma)$ corresponding to $(W_{n+k}, R_{n+k})$.  Thus, if $(C(X), \sigma)$ has the lifting property, where $\sigma$ denotes both the homeomorphism of $X$ and the isomorphism of $C(X)$ induced by $\sigma$, $(C(X), \sigma)^U$ cannot have the lifting property.
\end{proof}

\begin{cor}
The lifting property is not axiomatizable over $T_{Cantor-A}$.
\end{cor}

\subsection{Omitting types}
\begin{thm}\label{thm:OmittingTypes}
There is a countable sequence of types such that a model of $T_{Cantor-A}$ omits these types if and only if it satisfies the lifting property.
\end{thm}

\begin{proof}
A Cantor system $(X, \sigma)$ does not have the lifting property if and only if there is $n \in \mathbb{Z}^+$, $\alpha : (X,\sigma) \rightarrow (W_n,R_n)$, and a proper refinement $A$ of $A_\alpha$ such that $\forall k \in \mathbb{Z}^+$, if $\beta : (X,\sigma) \rightarrow (W_{n+k},R_{n+k})$ and $A_\beta$ is a proper refinement of $A$, then $\alpha \neq \xi \circ \beta$.  We will construct a collection of formulas, each describing a possible way in which the above statement can be realized, so that the omission of this collection of formulas is equivalent to the lifting property.

For any positive integer $n$, let $f(n) = 6^n(2n! + 2n - 1)$.  Let $n,m,k \in \mathbb{N}$ such that $f(n) < m < f(n + k)$.  Let $\varphi_1(x_1,\ldots,x_{f(n)})$ be the formula indicating that $x_1,\ldots,x_{f(n)}$ are projections, the product of any two is $0$, and that the homeomorphism $\sigma$ relates the supports of these projections in the same way that the elements of $W_n$ are related by $R_n$.  Let $\varphi_2(x_1,\ldots,x_{f(n)},y_1,\ldots,y_m)$ say $y_1,\ldots,y_m$ are projections, the product of any two is $0$, and that $y_1,\ldots,y_m$ refine $x_1,\ldots,x_{f(n)}$.  It does this with a disjunction over all possible ways $y_1,\ldots,y_m$ might refine $x_1,\ldots,x_{f(n)}$.  For example, one term in this disjunction is
\[\max \left\{ \norm{y_1 -x_1}, \ldots, \norm{y_{f(n)-1} - x_{f(n)-1}}, \norm{(\sum_{i=f(n)}^{m} y_i)  - x_{f(n)}} \right\}.\]  
Thus, $\max \{ \varphi_1, \varphi_2 \}$ says that $x_1,\ldots,x_{f(n)}$ are projections whose supports are the image of a continuous function from $(W_n, R_n)$, and are refined by the supports of $y_1,\ldots,y_m$.  Let $\varphi_3$ say that $z_1,\ldots,z_{f(n+k)}$ are projections whose supports are the image under a continuous function of $(W_n,R_n)$.  Let $\varphi_4(x_1,\ldots,x_{f(n)},z_1,\ldots,z_{f(n+k)})$ say that the supports of $z_1,\ldots,z_{f(n+k)}$ collapse onto the supports of $x_1,\ldots,x_{f(n)}$ with a disjunction over all possible ways this can happen (each disjunct corresponding to the different ways a spiral in $W_{n+k}$ could collapse onto any of the spirals in $W_n$ as dictated by the canonical map $\xi$).  Let $\varphi_5(y_1,\ldots,y_m,z_1,\ldots,z_{f(n+k)})$ say $z_1,\ldots,z_{f(n+k)}$ refines $y_1,\ldots,y_m$ with a disjunction over all possible ways this can happen.  Define 
\[\psi_{n,m,k}(x_1,\ldots,x_{f(n)},y_1,\ldots,y_m) = \max \left\{ \varphi_1, \varphi_2, \sup_{z_1, \ldots, z_{f(n+k)} \in \operatorname{proj}}(1 \dot{-}\max \{ \varphi_3,\varphi_4,\varphi_5 \})\right\}.\]  
Then $\psi_{n,m,k}$ says that there is $\alpha : (X,\sigma) \rightarrow (W_n,R_n)$ and a proper refinement $A$ of $A_\alpha$, where $A$ consists of $m$ clopen sets, such that if $\beta : (X,\sigma) \rightarrow (W_{n+k},R_{n+k})$ and $A_\beta$ is a proper refinement of $A$, then $\alpha \neq \xi \circ \beta$.

Let $\Sigma_{n,m} = \{ \psi_{n,m,k} \vert m <f(n+k) \}$, for $f(n) < m$.  Then a model satisfies the lifting property if and only if the model omits each $\Sigma_{n,m}$.  Therefore, all that remains is to show that each $\Sigma_{n,m}$ is indeed a type.  By Lemma 3.2, if $\alpha:(X,\sigma)\rightarrow(W_n,R_n)$ is a continuous surjection, then there is a sequence of partitions $(P_i)_{i\in\mathbb{N}}$, each the containing the same number of clopen sets, such that $\forall i \in \mathbb{N}, P_i$ refines $A_\alpha$, and there is no positive integer k such that for infinitely many $i \in \mathbb{N}$, $\exists\beta_i : (X,\sigma)\rightarrow (W_{n+k},R_{n+k})$ such that ${A_\beta}_i$ refines $P_i$ and $\alpha = \xi \circ \beta_i$.  In the proof of lemma 3.2 we have that $\vert P_i \vert = \vert A_\alpha \vert + 1$, but in fact $\vert P_i \vert$ can be made to have any size $m$ such that $m > f(n)$.  Thus, the ultrapower constructed in the proof of Theorem \ref{thm:NonaxiomatizableLifting}  will satisfy $\Sigma_{n,m}$.  Thus, each $\Sigma_{n,m}$ is a type.
\end{proof}

\begin{cor}\label{cor:NotSepCategoricalLifting}
Let $T = \op{Th}((C(2^\mathbb{N}), \sigma))$, where $\sigma$ is a generic homeomorphism of the Cantor set.  Then $T$ is not separably categorical.
\end{cor}
\begin{proof}
By Theorem \ref{thm:NonaxiomatizableLifting} there are models of $T$ that do not omit the types from Theorem \ref{thm:OmittingTypes}.  If $\mathcal{M}$ is such a model and $\overline{a} \in \mathcal{M}$ realizes one of the types from Theorem \ref{thm:OmittingTypes} then any separable elementary substructure of $\mathcal{M}$ containing $\overline{a}$ will be a separable model of $T$ that is not isomorphic to $(C(2^{\mathbb{N}}), \sigma)$.
\end{proof}

We note that in Corollary \ref{cor:NotSepCategoricalLifting} the separable models obtained will be isomorphic to ones of the form $C(2^\mathbb{N}, \tau)$ for some $\tau$, so Corollary \ref{cor:NotSepCategoricalLifting} produces homeomorphisms of the Cantor set that are quite similar to the generic homeomorphism, but are not topologically conjugate to it.

\begin{cor}\label{thm:Prime}
If $\sigma$ denotes the generic homeomorphism of the Cantor set $2^\mathbb{N}$, then $(C(2^\mathbb{N}), \sigma)$ is the prime model of its theory.
\end{cor}

\begin{proof}
Let $M = (\op{CL}(2^{\mathbb{N}}, \sigma_B))$, where $\sigma_B$ is the automorphism of the clopen algebra of the Cantor set induced by the homeomorphism $\sigma$.  Notice that each formula $\psi_{n,m,k}$ from the proof of Theorem \ref{thm:OmittingTypes} refers only to projections, and so it has a corresponding formula in the discrete language of Boolean algebras.  We thus have that $M$ is the unique countable model of its theory omitting the types (in discrete logic) corresponding to the ones from Theorem \ref{thm:OmittingTypes}.  If a finite tuple $\overline{a} \in M$ satisfies a non-principal type in $M$ then by the Omitting Types Theorem for discrete logic we could build a countable structure $N$ such that $N \equiv M$ and in $N$ both $\op{tp}(\overline{a})$ and all of the types from Theorem \ref{thm:OmittingTypes} are omitted.  Since the types from Theorem \ref{thm:OmittingTypes} are omitted in $N$ and $N$ is countable we then have $N \cong M$, which is a contradiction since $N$ omits $\op{tp}(\overline{a})$ and $M$ realizes $\op{tp}(\overline{a})$.  Thus $M$ is the atomic, and hence prime, model of its theory.

It remains to show that $(CL(2^{\mathbb{N}}), \sigma_B)$ being prime implies that $(C(2^{\mathbb{N}}), \sigma)$ is prime.  Via Gel'fand duality we immediately obtain that $(C(X),\sigma)$ embeds into every model of its theory via an embedding that is elementary on projections, so what remains is to show that such an embedding is elementary.

Suppose $(C(2^{\mathbb{N}}),\tau) \equiv (C(2^{\mathbb{N}})\sigma)$, and let $h: (C(2^{\mathbb{N}}),\sigma) \rightarrow (C(2^{\mathbb{N}}),\tau)$ be an embedding that is elementary on projections.  Let $M=(C(2^{\mathbb{N}}),\sigma)$ and $N=(C(2^{\mathbb{N}}),\tau)$.  We first show that $h$ is elementary on linear combinations of projections.  Let $\bar{a} \in M$ be a tuple of linear combinations of projections, and let $\varphi$ be a formula.  Then there is a formula $\psi$ such that $\psi$ accepts as arguments the tuple of projections making up $\bar{a}$, call these projections $\bar{b}$, and $\varphi( \bar{a} )$ is equivalent to $\psi( \bar{b} )$.  Then we have that $\varphi^M(\bar{a}) = \psi^M(\bar{b}) = \psi^N(h(\bar{b})) = \varphi^N(h(\bar{a}))$, so h is elementary on linear combinations of projections.

We now use the fact that every element of $C(X)$ is the limit of linear combinations of projections.  Let $\bar{a}$ be a tuple of elements of $M$ and let $\varphi$ be a formula.  Then there exists a sequence $(\bar{b}_n)_{n \in \mathbb{N}}$ such that each $\bar{b}_n$ is a tuple of linear combinations of projections and $\bar{a}$ is the limit of $(\bar{b}_n)_{n \in \mathbb{N}}$.  Then by the continuity of $\varphi$ and $h$, $\varphi^M(\bar{a}) = \lim_n\varphi^M(\bar{b}_n) = \lim_n\varphi^N(h(\bar{b}_n)) = \varphi^N(h(\bar{a}))$.  Thus, $h$ is elementary, as desired.
\end{proof}

\begin{rem}
The proof of Corollary \ref{thm:Prime} relies on being able to simultaneously omit several types, each of which is known to be omissible individually.  In discrete first-order logic the Omitting Types Theorem guarantees that this is possible, but in continuous logic joint omission of types is a significantly more subtle matter (see \cite{Farah2018}).  It is for this reason that we have resorted to using the discrete model theory of Boolean algebras in the proof of Corollary \ref{thm:Prime}.
\end{rem}

\section{A Fra\"iss\'e construction}
If $\sigma$ is a homeomorphism of the Cantor set that has the lifting property, then $\sigma$ is \emph{generic}, in the sense that the conjugacy class of $\sigma$ is comeagre in the group $\mathcal{H}(2^\mathbb{N})$ of homeomorphisms of the Cantor set.  In this section we will be discussing another notion of genericity, so to avoid confusion, in this section we will refer to homeomorphisms with the lifting property as \emph{special}.

In model theory the Fra\"iss\'e limit of a Fra\"iss\'e class is also often referred to as a \emph{generic} structure.  The two usages are related, as Kabluchko and Tent have shown that the Fra\"iss\'e limit of a Fra\"iss\'e class $\mathcal{C}$ (of discrete structures) is the unique countable structure whose isomorphism class is comeagre in the space of structures with age contained in $\mathcal{C}$ \cite{KabluchkoTent}.  

Fra\"iss\'e constructions have been adapted to continuous logic in \cite{SchoretsanitisFraisse} and \cite{BenYaacov2013}, and specifically for C*-algebras in \cite{EagleFraisse}.  In particular, it is shown in \cite{EagleFraisse} that the C*-algebra $C(2^{\mathbb{N}})$ is the Fra\"iss\'e limit of the class of finite-dimensional commutative C*-algebras.  Curiously, we can use this fact to construct a homeomorphism of the Cantor set which arises as a Fra\"iss\'e limit, but which is not special, thus showing that these two possible meanings of the term ``generic homeomorphism" are not the same.  This is of particular interest because the first construction of the special homeomorphism (in \cite{KechrisRosendal}) also involved a Fra\"iss\'e construction (though that construction is more complicated than the one described here).

For the rest of this section, let $\mathcal{K}$ denote the class of structures of the form $(C(X), \tau)$ where $C(X)$ is finite-dimensional and $\tau$ is an automorphism of $C(X)$.

\begin{thm}
The class $\mathcal{K}$ is a Fra\"iss\'e class.  The Fra\"iss\'e limit of $\mathcal{K}$ is of the form $(C(2^\mathbb{N}), \sigma)$, where $\sigma$ corresponds to a homeomorphism of $2^\mathbb{N}$.
\end{thm}

\begin{proof}
It is clear that $\mathcal{K}$ satisfies the hereditary property.

We now show that $\mathcal{K}$ has the joint embedding property.  By Gel'fand duality, every finite dimensional commutative $C^*$-algebra may be thought of as the functions from some finite discrete set to the complex numbers, and every automorphism corresponds to a permutation of this finite set.  Moreover, for finite discrete sets $X$ and $Y$, every embedding from $C(X)$ to $C(Y)$ arises from a surjection from $Y$ onto $X$.  Thus, to show that $\mathcal{K}$ has the joint embedding property, it is equivalent to show that for finite sets $X$ and $Y$ and permutations $\sigma : X \rightarrow X$ and $\tau : Y \rightarrow Y$, there is a finite set $Z$ and permutation $\pi : Z \rightarrow Z$ along with surjections $f: Z \rightarrow X$ and $g : Z \rightarrow Y$ such that $f \circ \pi = \sigma \circ f$ and $g \circ \pi = \tau \circ g$.

Let $X = \{ 1,\ldots,n \}$ and let $Y = \{ 1,\ldots,m \}$.  Let $Z = X \times Y$.  Define $\pi : Z \rightarrow Z$ by $\pi (j,k) = (\sigma(j), \tau(k))$, define $f: Z \rightarrow X$ by $f(j,k) = j$, and define $g: Z \rightarrow Y$ by $g(j,k) = k$.  Then it is easy to check these functions have the desired properties.

We now show $\mathcal{K}$ has the amalgamation property.  Just as above, it is easiest to do this in the context of permutations of finite discrete sets.  Suppose we have finite discrete sets with permutations $(W,\sigma)$, $(X,\tau)$, and $(Y,\pi)$, and surjections $f: X \rightarrow W$ and $g: Y \rightarrow W$ such that $f \circ \tau = \sigma \circ f$ and $g \circ \pi = \sigma \circ g$.  We must find $Z$ and $\rho: Z \rightarrow Z$ and surjections $h: Z \rightarrow X$ and $i: Z \rightarrow Y$ such that $h \circ \rho = \tau \circ h$, $i \circ \rho = \pi \circ i$, and $f \circ h = g \circ i$.

We first consider the case where $\sigma$ is a cycle.  Note that here we consider any point left unchanged by the permutation to be a cycle of length 1, so when we say $\sigma$ is a cycle we mean that $\sigma$ cycles through all the elements of $W$.  Then both $\tau$ and $\pi$ are composed of disjoint cycles, the length of each such cycle divisible by the length of $\sigma$ so that a projection onto $(W,\sigma)$ is possible.  This is because of the fact that in order for the condition $f \circ \tau = \sigma \circ f$ to be satisfied, we must have that $f$ projects each cycle $c$ in $\tau$ onto $\sigma$ in such a way that $c$ ``wraps around" $\sigma$ the appropriate number of times, and similarly for $\pi$ and $g$.

Let $n$ denote the maximum of the number of distinct cycles in $\tau$ and the number of distinct cycles in $\pi$.  Let $m$ denote the lowest common multiple of the length of all the cycles in $\tau$ and $\pi$.  We let $Z$ consist of $nm$ points and $\rho$ be $n$ disjoint cycles of length $m$.  Then for each cycle $c$ in $\tau$ we may project a cycle $d$ from $\rho$ onto $c$, since the length of $c$ divides the length of $d$ so we may therefore ``wrap" $d$ around $c$ the appropriate number of times.  We therefore define $h$ to be such a surjection, so that $h \circ \rho = \tau \circ h$.  We then construct $i$ in a similar way so that $i \circ \rho = \pi \circ i$, although we must be careful to choose $i$ so that $f \circ h = g \circ i$.  We do this as follows.  For $z \in Z$, $z$ must appear in some cycle $d$ in $\rho$.  By our choice of $n$ and $m$ we may let $i$ project $d$ onto any cycle in $\pi$, however we require that $z$ is mapped to an element of $g^{-1} (f \circ h ( \{ z \} ))$.  We then have that $f \circ h = g \circ i$.

We now consider the more general case that $\sigma$ is a product of disjoint cycles.  We can then decompose $W$ into a disjoint subsets $W_1,\ldots,W_n$ such that $\sigma$ is a cycle on each of these components.  We then repeat the construction above on each of these pieces and put the resulting sets together to get $Z, \ \rho, \ h$, and $i$ as required.

As shown in \cite{EagleFraisse}, the class of reducts of members of $\mathcal{K}$ to the language of C*-algebras is a Fra\"iss\'e class with Fra\"iss\'e limit $C(2^{\mathbb{N}})$.  Thus the Fra\"iss\'e limit of $\mathcal{K}$ is of the form $(C(2^{\mathbb{N}}), \sigma)$ for some automorphism $\sigma$ of $C(2^{\mathbb{N}})$.
\end{proof}

Our next result requires the notion of a wandering point of a homeomorphism.

\begin{defn}
Let $X$ be a topological space, and let $\sigma : X \to X$ be a homeomorphism.  A point $x \in X$ is called \emph{wandering} if there is an open $U \subseteq X$ such that $x \in U$ and an $N \in \mathbb{N}$ such that for all $n > N$, $\sigma^n(U) \cap U = \emptyset$.
\end{defn}

\begin{thm}
The homeomorphism arising from the Fra\"iss\'e limit of $\mathcal{K}$ is not special.
\end{thm}
\begin{proof}
We will refer to the homeomorphism arising from the Fra\"iss\'e construction as $\sigma$ and the special homeomorphism as $\tau$.  We show that we can distinguish between these two maps by establishing the existence of wandering points for $\tau$ and the nonexistence of wandering points for $\sigma$.

First we show that $\tau$ has wandering points.  By the lifting property, we can find a continuous surjection $\phi : (2^{\mathbb{N}}, \tau) \rightarrow (W_2, R_2)$.  Choose one finite spiral in $W_2$ and describe it as $(\{l \} \times \{0, 1 \}) \cup (\{m \} \times \{-1, 0, 1 \}) \cup (\{r \} \times \{0, 1 \})$.  Then $V = \phi^{-1} \{ (m,1) \}$ is a wandering set, since for all $n \geq 1$, $\tau ^n (V) \subseteq \phi^{-1} \{ (r, k) \}$, where $k = 0$ if $n$ is odd and $k = 1$ if $n$ is even.  In either case, $V \cap \tau^n(V) = \emptyset$, so every point in $V$ is wandering.

We now show that $\sigma$ has no wandering points.  Since the Cantor set is totally disconnected, it has a basis of clopen sets, and it is therefore enough to show there are no wandering clopen sets.  Since projections are precisely the indicator functions of clopen sets, we prove the nonexistence of wandering points by showing that for any projection $p \in C(2^{\mathbb{N}})$ there is a positive integer $n$ such that the support of $\sigma^n(p)=p$.

Let $p$ be any projection in $C(2^{\mathbb{N}})$.  $(C(2^{\mathbb{N}}), \sigma)$ is a direct limit of the elements of $\mathcal{K}$, and so $p$ is an equivalence class of vectors of the form $v \in (\mathbb{C}^m, \rho)$, for $m$ ranging over the positive integers and $\rho: \mathbb{C}^m \rightarrow \mathbb{C}^m$ an automorphism.  Choose some representative $v$ from this equivalence class.  Then $v \in (\mathbb{C}^m, \rho)$  for some $m$ and $\rho$.  The automorphism $\rho$ acts on elements of $\mathbb{C}$ by permuting the coordinates, so there exists a positive integer $n$ such that $\rho^n$ is the identity.  Thus, $\rho^n(v) = v$.  Since the automorphisms of the elements of $\mathcal{K}$ commute with the homomorphisms connecting them, every element of the equivalence class $p$ must return to itself after $n$ iterations of its automorphism.  Since this is true of every element of the equivalence class, we must have that $\sigma^n(p) = p$, and hence $\sigma$ has no wandering points.
\end{proof}
% ----------------------------------------------------------------
\bibliographystyle{amsalpha}
\bibliography{Cantor}
\end{document}